\documentclass[a4paper,11pt,reqno]{article}

\usepackage[utf8]{inputenc}
\usepackage[T1]{fontenc}
\usepackage{lmodern}
\usepackage[english]{babel}
\usepackage{microtype}

\usepackage{amsmath,amssymb,amsfonts,amsthm}
\usepackage{mathtools,accents}
\usepackage{mathrsfs}
\usepackage{aliascnt}
\usepackage{braket}
\usepackage{bm}

\usepackage[a4paper,margin=3cm]{geometry}
\usepackage[citecolor=blue,colorlinks]{hyperref}

\usepackage{enumerate}
\usepackage{xcolor}

\makeatletter
\g@addto@macro\@floatboxreset\centering
\makeatother

\makeatletter
\def\newaliasedtheorem#1[#2]#3{
  \newaliascnt{#1@alt}{#2}
  \newtheorem{#1}[#1@alt]{#3}
  \expandafter\newcommand\csname #1@altname\endcsname{#3}
}
\makeatother

\numberwithin{equation}{section}

\newtheoremstyle{slanted}{\topsep}{\topsep}{\slshape}{}{\bfseries}{.}{.5em}{}

\theoremstyle{plain}
\newtheorem{theorem}{Theorem}[section]
\newaliasedtheorem{proposition}[theorem]{Proposition}
\newaliasedtheorem{lemma}[theorem]{Lemma}
\newaliasedtheorem{corollary}[theorem]{Corollary}
\newaliasedtheorem{conjecture}[theorem]{Conjecture}
\newaliasedtheorem{counterexample}[theorem]{Counterexample}

\theoremstyle{definition}
\newaliasedtheorem{definition}[theorem]{Definition}
\newaliasedtheorem{question}[theorem]{Question}
\newaliasedtheorem{example}[theorem]{Example}

\theoremstyle{remark}
\newaliasedtheorem{remark}[theorem]{Remark}

\let\altphi\phi
\let\phi\varphi
\let\varphi\altphi
\let\altphi\undefined

\newcommand{\di}{\mathop{}\!\mathrm{d}}

\DeclareMathOperator{\supp}{supp}

\newcommand{\haus}{\mathscr{H}}

\newcommand{\dist}{\mathsf{d}}

\DeclareMathOperator{\RCD}{RCD}
\DeclareMathOperator{\CD}{CD}

\newfont{\tmpf}{cmsy10 scaled 2500}

\newcommand{\intav}{{\mathop{\int\kern-10pt\rotatebox{0}{\textbf{--}}}}}
\newcommand{\intavv}{{\mathop{\int\kern-8pt\rotatebox{0}{\textbf{--}}}}}

\renewcommand{\ }{\text{ }}
\def\<{\langle}
\def\>{\rangle}

\begin{document}

\title{Gap phenomena under curvature restrictions}
 
\author{
Shouhei Honda\thanks{Graduate School of Mathematical Sciences, The University of Tokyo, \url{shouhei.@ms.u-tokyo.ac.jp}} \,\,and Andrea Mondino\thanks{Mathematical Institute, University of Oxford, \url{andrea.mondino@maths.ox.ac.uk}}
}

\maketitle
\abstract{In the paper we discuss  gap phenomena of three different types related to Ricci (and sectional) curvature.
The first type is about \textit{spectral gaps}. The second type is about \textit{sharp gap metric-rigidity}, originally due to Anderson. The third is about \textit{sharp gap topological-rigidity}.  We also propose open problems along these directions. 
\section{Spectral gap I: Poincar\'e inequality for functions}
Throughout the paper, $M^n$ will denote a complete Riemannian manifold of dimension $n$ without boundary. Moreover we follow standard notations, including $\mathrm{Ric}=\mathrm{Ric}(M^n), \mathrm{diam}=\mathrm{diam}(M^n)$, and $\mathrm{vol}=\mathrm{vol}(M^n)$, in Riemannian geometry.

In this section, under the standing assumption that $M^n$ is compact, we are interested in the best constant $C$, denoted by $C_P=C_P(M^n)$, in the Poincar\'e inequality for functions:
\begin{equation}
\int_{M^n}\left|f-\intav_{M^n}f\di \mathrm{vol}\right|^2\di \mathrm{vol}\le C\int_{M^n}|\nabla f|^2\di \mathrm{vol},\quad \forall f \in C^{\infty}(M^n),
\end{equation}
where $\intavv_{M^n}$ denotes the average over $M^n$. Note that $C_P$ coincides with the inverse of the first positive eigenvalue, denoted by $\mu=\mu(M^n)$, of the positive Laplacian $-\Delta$:
\begin{equation}
C_P=\frac{1}{\mu}.
\end{equation}
Let us recall results by Lichnerowicz \cite{Lic},  Li-Yau \cite{LY} and Buser \cite{Buser}.
\begin{theorem}[Lichnerowicz, Li-Yau and Buser]\label{BLY}
If 
\begin{equation}
\mathrm{Ric} \ge -(n-1),\quad \mathrm{diam} \le d
\end{equation}
then
\begin{equation}
\mu \ge c(n, d)>0, 
\end{equation}
or, equivalently,
\begin{equation}
C_P\le \frac{1}{c(n, d)}.
\end{equation}
\end{theorem}
\begin{proof}
In order to keep the presentation as short as possible, let us provide a proof in the special case $\mathrm{Ric}\ge n-1$ (then we know $\mathrm{diam} \le \pi$ by a theorem of Meyers, namely the dependence of $d$ can be dropped in this case), originally due to Lichnerowicz.

Denoting by $f$ an eigenfunction of $-\Delta$ relative to the first  eigenvalue $\mu$,  Bochner's formula gives
\begin{align}
    \frac{1}{2}\Delta |\nabla f|^2&=|\mathrm{Hess}_f|^2+\langle \nabla \Delta f, \nabla f\rangle + \mathrm{Ric}(\nabla f, \nabla f) \nonumber \\
    & \ge \frac{(\Delta f)^2}{n}-\mu |\nabla f|^2 +(n-1)|\nabla f|^2, \label{eq:proofmu}
\end{align}
where we used the Cauchy-Schwarz inequality $|\mathrm{Hess}_f| \ge |\Delta f|/\sqrt{n}$.
Integrating \eqref{eq:proofmu} over $M^n$ yields the sharp estimate:
\begin{equation}
    \mu \ge n.
\end{equation}
\end{proof}
Note that such a Poincar\'e inequality (and generalizations) is valid even for non-smooth metric measure spaces with Ricci curvature bounded below and dimension bounded above in a synthetic sense, i.e., the so-called $\CD(K,N)$ spaces of Lott-Sturm-Villani (see the works by Sturm \cite{Sturm06b}, Lott-Villani \cite{LVJFA, LottVillani}, Rajala \cite{Rajala},  Cavalletti and the second named author \cite{CM1, CM2}).

\section{Spectral gap II: Poincar\'e inequality for differential forms}
In this section, under the standing assumption that $M^n$ is compact, we deal with a Poincar\'e inequality for differential forms with respect to the Hodge energy. Namely, we focus on the best constant $C>0$, denoted by $C_{P, k}=C_{P, k}(M^n)$, such that
\begin{equation}
\int_{M^n}\left| \omega-\eta_{\omega}\right|^2\di \mathrm{vol} \le C\int_{M^n} \left( |d\omega|^2 +(\delta \omega)^2\right)\di \mathrm{vol}, \quad \forall \omega \in \Gamma\left(\bigwedge^kT^*M^n\right),
\end{equation}
where $\eta_{\omega}$ is the harmonic part of $\omega$. Note that $C_{P, k}$ coincides with the inverse of the first positive eigenvalue $\mu_k=\mu_k(M^n)$ of the Hodge Laplacian $\Delta_{H, k}=\delta d+d\delta$ acting on differential $k$-forms:
\begin{equation}
C_{P, k}=\frac{1}{\mu_k}.
\end{equation}
Firstly let us provide an easy consequence of Theorem \ref{BLY} on differential forms of  top degree. It should be emphasized that there is no assumption on the volume.
\begin{proposition}\label{orient}
    If
    \begin{equation}
\mathrm{Ric} \ge -(n-1),\quad \mathrm{diam} \le d,
\end{equation}
then 
\begin{equation}
    C_{P,n}\le C(n, d).
\end{equation}
\end{proposition}
\begin{proof}
    It is trivial that the conclusion holds if $M^n$ is orientable because of the Hodge star operator and Theorem \ref{BLY}. 
    
     If $M^n$ is not orientable,  take the Riemannian double cover $\pi:\tilde M^n \to M^n$. For any eigen-$n$-form $\omega$ on $M^n$, consider $\pi^*\omega$ on $\tilde M^n$ which is also an eigen-$n$-form of the same eigenvalue. Thus applying the conclusion in the orientable case to $\pi^*\omega$  completes the proof.
\end{proof}
Secondly, let us recall a result of Colbois-Courtois in \cite{CC} giving an estimate on $C_{P, k}$, under the assumption of bounded sectional curvature.
\begin{theorem}[Colbois-Courtois]\label{cc}
If
\begin{equation}\label{sectass}
|\mathrm{sec}|\le 1,\quad \mathrm{diam} \le d,\quad \mathrm{vol} \ge v,
\end{equation}
then for any $k \ge 1$
\begin{equation}
\mu_k \ge c(n, k, d, v),
\end{equation}
or, equivalently, 
\begin{equation}
C_{P, k} \le \frac{1}{c(n, k, d, v)}.
\end{equation}
\end{theorem}
\begin{proof}
The proof is done by a contradiction argument. If it is not the case, then we can find a sequence of closed manifolds $M^n_i$ with (\ref{sectass}) and $\mu_k(M^n_i) \to 0$. Thanks to \cite{peters} by Peters, with no loss of generality we can assume that $M_i^n$ $C^{1, \alpha}$-converge to a $C^{1, \alpha}$-Riemannian manifold $M^n$ for any $0 < \alpha<1$ (in particular $M_i^n$ is diffeomorphic to $M^n$ for any sufficiently large $i$). Since the spectral convergence for $\Delta_{H, k}$ is trivially satisfied for this sequence, we have
\begin{equation}
\limsup_{i \to \infty}b_k(M_i^n)<b_k(M^n)
\end{equation}
because of the Hodge theory on $M^n$. This contradicts the fact that $M_i^n$ is diffeomorphic to $M^n$. 
\end{proof}
We refer the reader to the work  \cite{M} by Mantuano, for explicit bounds on  $c(n, k, d, v)$.
It is worth mentioning that a lower bound on the volume in the assumption of the theorem above cannot be dropped because of an example in \cite{CC}.

From the result above, it is natural to ask the following.
\begin{question}\label{secricci}
What happens in Theorem \ref{cc} if we replace $\mathrm{sec}$ by $\mathrm{Ric}$?
\end{question}
Note that the case of dimension at most $3$ can be covered by Theorem \ref{cc}. Thus a new situation appears in dimension $4$.

We are now in a position to introduce a main result of \cite{HM} giving an answer to Question \ref{secricci} in dimension $4$.
\begin{theorem}\label{eigenmain}
If
\begin{equation}\label{eq:43dv}
n=4,\quad |\mathrm{Ric}|\le 3,\quad \mathrm{diam} \le d,\quad \mathrm{vol} \ge v,
\end{equation}
then
\begin{equation}\label{eq:mu1geqcdv}
\mu_1 \ge c(d, v),
\end{equation}
or, equivalently,
\begin{equation}
C_{P, 1} \le \frac{1}{c(d, v)}.
\end{equation}
\end{theorem}

\begin{proof}[Sketch of the proof]
 As in Theorem \ref{cc}, the proof is performed via a contradiction argument, exploiting the pre-compactness of the family of Riemannian manifolds satisfying \eqref{eq:43dv} and the regularity theory of the arising limits.  

 Assume by contradiction that \eqref{eq:mu1geqcdv} does not hold.
 Then we can find a sequence of closed Riemannian manifolds $M_i^4$ with $|\mathrm{Ric}| \le 3$, $\mu_1(M_i^4) \to 0$ and
 \begin{equation}
 M^4_i \stackrel{\mathrm{GH}}{\to} X^4
 \end{equation}
 for some non-collapsed Ricci limit space $X^4$. From \cite{SormaniWei} by Sormani-Wei, it follows that $b_1(M_i^4)=b_1(X^4)$ for any sufficiently large $i$. Thanks to works by Anderson, and Cheeger-Naber \cite{Anderson, Anderson2, CheegerNaber}, we know that $X^4$ is an orbifold. Then the spectral convergence established in \cite{Hondas} by the first named author together with $\mu_1(M^4_i) \to 0$ allows us to conclude
 \begin{equation}
\limsup_{i \to \infty}b_1(M_i)<b_1(X^4) 
\end{equation}
because of both Hodge theories \cite{Satake, Gigli} by Satake and Gigli for orbifolds and RCD spaces, respectively. Thus we have a contradiction.
\end{proof}
\begin{corollary}\label{cor:spectralgap4}
    Under the same assumptions of Theorem \ref{eigenmain}, it holds
    \begin{equation}
        C_{P, k} \le C(d, v), \quad \forall k=0, 1, 2, 3, 4.
    \end{equation}
\end{corollary}

\begin{proof}
    Arguing as in the  proof of Proposition \ref{orient}, we obtain the result in the case when $k=3$ (recall that we already obtained the results in the case when $k=0, 1$ and $4$). Thus it is enough to check the assertion for $k=2$ when 
 $M^4$ is oriented (by the same reason as above).  To this aim, the Hodge decomposition yields the desired spectral gap of the Hodge Laplacian for $2$-forms from that of $1$-forms (see, for instance, the discussion around \cite[(2.1)]{Takahashi}. Thus we conclude. 


\end{proof}

Let us recall the following conjecture proposed in \cite{HM}.
\begin{conjecture}\label{specgap}
If
\begin{equation}\label{riccibound}
\mathrm{Ric}\ge -(n-1),\quad \mathrm{diam} \le d,\quad \mathrm{vol} \ge v,
\end{equation}
then
\begin{equation}
C_{P, 1} \le C(n, d, v).
\end{equation}
\end{conjecture}
\noindent
Two sub-conjectures already of interest would be:

\begin{itemize}
\item Establish Conjecture \ref{specgap} under the stronger assumption that $\mathrm{Sec}\geq -1$ (see also Conjecture \ref{section} for a stronger statement);
\item  Establish Conjecture \ref{specgap} under the stronger assumption that $\mathrm{|Ric|}\leq n-1$;
\end{itemize}

We are able to give a partial answer to the conjecture. 
\begin{proposition}\label{dim2poincare}
Conjecture \ref{specgap} holds in dimension $2$. More strongly, the dependence on $v>0$ can be dropped;   namely, if 
\begin{equation}
n=2,\quad \mathrm{Ric} \ge -1,\quad \mathrm{diam} \le d,
\end{equation}then
\begin{equation}\label{2dimpoincare}
    C_{P, 1}\le C(d).
\end{equation}

\end{proposition}
\begin{proof}
If $M^2$ is orientable,  classical Hodge theory in dimension 2 (see for instance \cite[Proposition 2.4]{Takahashi} by Takahashi) implies that  $C_{P,1}=C_P$. Thus  (\ref{2dimpoincare}) is a direct consequence of Theorem \ref{BLY}.

If $M^2$ is not orientable, one can take the orientable double cover and argue as in the proof of Proposition \ref{orient}.    
\end{proof}

\section{Gap metric-rigidity for Einstein manifolds}
In this section, we deal with gap theorems for Ricci-flat and positive Einstein manifolds. 

\subsection{Anderson's gap theorem for Ricci-flat manifolds}
If $\mathrm{Ric}\ge 0$, then the asymptotic volume ratio (AVR) of $M^n$ is defined by
\begin{equation}
\mathrm{AVR}=\mathrm{AVR}(M^n):=\lim_{r\to \infty}\frac{\mathrm{vol}(B_r(x))}{\omega_nr^n} \in [0, 1],
\end{equation}
where the existence of the limit comes from the Bishop-Gromov inequality (moreover it does not depend on the choice of $x \in M^n$) and the upper bound $1$ is a direct consequence of the Bishop inequality.

Let us start this section by recalling the following well known fact.
\begin{proposition}\label{rigid}
Assume $\mathrm{Ric} \ge 0$. Then $M^n$ is isometric to $\mathbb{R}^n$ if and only if $\mathrm{AVR}=1$. 
\end{proposition}
\begin{proof}
We only discuss the ``if'' implication, as the converse is trivial.
 If $\mathrm{AVR}=1$, then the Bishop and the Bihop-Gromov inequalities show
\begin{equation}
\mathrm{vol} (B_r(x))=\omega_nr^n,\quad \forall r>0,\quad \forall x \in M^n.
\end{equation}
This easily implies that any ball in $M^n$ is isometric to a ball of the same radius in $\mathbb{R}^n$. The conclusion follows. 
\end{proof}

From this proposition, it is natural to ask what happens if $\mathrm{AVR}$ is close to $1$. The following gives an answer to the question.

\begin{theorem}[Colding, Cheeger-Colding]\label{ChC}
The following holds:
\begin{enumerate}
\item Assume $\mathrm{Ric} \ge 0$. Then $M^n$ is isometric to $\mathbb{R}^n$ if and only if a tangent cone at infinity of $M^n$ is isometric to $\mathbb{R}^n$.
\item For any $n \ge 2$, there exists $\epsilon(n)>0$ such that if $\mathrm{Ric}\ge 0$ and $\mathrm{AVR} \ge 1-\epsilon(n)$, then $M^n$ is diffeomorphic to $\mathbb{R}^n$.
\end{enumerate}
\end{theorem}
\begin{proof}[Sketch of the proof]
(1) is a direct consequence of Proposition \ref{rigid} and the volume convergence established in \cite{CheegerColding1, Colding} by Cheeger-Colding and Colding.

For (2),
under the assumptions, we know that any ball of radius $R>0$ in $M^n$ is $(\delta_nR)$-Gromov-Hausdorff close to a ball of radius $R>0$ in $\mathbb{R}^n$. The conclusion follows from the intrinsic Reifenberg method established in \cite{CheegerColding1}.
\end{proof}
From now on let us focus on Ricci flat manifolds. The following is classical.
\begin{proposition}\label{gap3}
If 
\begin{equation}\label{eq:3AVR}
n \le 3,\quad \mathrm{Ric}\equiv 0,\quad \mathrm{AVR}>0,
\end{equation}
then $M^n$ is isometric to $\mathbb{R}^n$.
\end{proposition}
\begin{proof}
If $n\leq 3$, then $\mathrm{Ric}\equiv 0$ implies that  $\mathrm{Sec}\equiv 0$, which in turn implies that $M^n$ is locally isometric to $\mathbb{R}^n$. This means that there is a discrete free subgroup $\Gamma$ of the isometry group of $\mathbb{R}^n$ such that $M$ is isometric to the quotient $\mathbb{R}^n/\Gamma$. Since $M^n$ is smooth, $\Gamma$ cannot contain rotations, which have fixed point and would create singularities in the quotient. Moreover, $\Gamma \setminus\{\mathrm{id}\}$ cannot consist only of a reflection, as in this case $\mathbb{R}^n/\Gamma$ is a half-space, contradicting that $M^n$ is without boundary. Thus, if $\Gamma \setminus\{\mathrm{id}\}$ is non-empty, then it must contain at least a non-trivial translation. This implies that $\mathrm{AVR}(\mathbb{R}^n/\Gamma)=0$,  contradicting \eqref{eq:3AVR}. We thus conclude that $\Gamma$ is trivial, yielding that $M^n$ is isometric to $\mathbb{R}^n$.
\end{proof}
We are now in position to introduce the main focus of this section, called Anderson's gap theorem. The following should be compared with Theorem \ref{ChC}.
\begin{theorem}[Anderson]\label{Andersongap}
For any $n \ge 2$, there exists $\epsilon(n)>0$ such that if $\mathrm{Ric}\equiv 0$ and $\mathrm{AVR} \ge 1-\epsilon(n)$, then $\mathrm{AVR}=1$, namely $M^n$ is isometric to $\mathbb{R}^n$.
\end{theorem}
\begin{proof}[Sketch of the proof]

\textbf{Step 1.} Both the injectivity and harmonic radii at some point are infinite if $\mathrm{AVR}$ is close to $1$, quantitatively.

The proof is done by a contradiction argument. For simplicity we focus only on the case of the harmonic radius.
If it is not the case, then there exists a sequence  $M^n_i$ of Ricci flat manifolds with  finite harmonic 
radius $r_i$ at some points $x_i \in M_i^n$ and $\mathrm{AVR} \to 1$. 
Consider the rescaled manifolds $r_i^{-1}M_i^n$ whose harmonic 
radius at $x_i$ is equal to $1$ by definition. After passing to a subsequence, we have for some pointed metric space $(X^n, x)$
\begin{equation}\label{ri}
    \left(r_i^{-1}M_i^n, x_i\right) \stackrel{\mathrm{pGH}}{\to} (X^n, x).
\end{equation}
Since $X^n$ is $n$-dimensional with non-negative Ricci and its $\mathrm{AVR}$ is equal to $1$, it follows that $X^n$ is isometric to $\mathbb{R}^n$.  Note that
the smooth regularity of $X^n$ is a direct consequence of the elliptic regularity theory together with the Einstein equation $\mathrm{Ric}=0$. Moreover the same reason on the regularity 
allows us to improve the convergence (\ref{ri}) to smooth convergence. In particular, we know that the harmonic 
radius of $X^n$ must be $1$ because so is $r_i^{-1}M_i^n$. This is a contradiction.
\smallskip

\textbf{Step 2.} Conclusion.

Thanks to \textbf{Step 1} together with \cite{AC} by Anderson-Cheeger, we can construct a globally bi-Lipschitz embedding harmonic map $\Phi:M^n \to \mathbb{R}^n$. Take a blow-down of $\Phi$;
\begin{equation}
    \tilde \Phi:C(Z) \to \mathbb{R}^n
\end{equation}
where $C(Z)$ is a tangent cone at infinity of $M^n$. Then $\tilde \Phi$ is a linear bi-Lipschitz embedding. Since $C(Z)$ is $n$-dimensional, this shows that $C(Z)$ is isometric to $\mathbb{R}^n$. Thus the conclusion follows from Theorem \ref{ChC}.
\end{proof}
Based on this result, let us define the Anderson constant as follows.
\begin{definition}
For any $n \ge 2$, define $C_{A, 0}(n)$ as the infimum of $\epsilon>0$ satisfying that if a complete Ricci flat manifold $M^n$ satisfies $\mathrm{AVR}>\epsilon$, then $M^n$ is isometric to $\mathbb{R}^n$. 
\end{definition}
The following question is natural.
\begin{question}\label{andconst}
Determine $C_{A, 0}(n)$ explicitly.
\end{question}
Note that Proposition \ref{gap3} yields 
\begin{equation}
C_{A, 0}(n)=0,\quad \text{if $n \le 3$.}
\end{equation}
Thus, we focus on the case when $n \ge 4$. The first simple observation is in dimension $4$. Though the following seems to be well-known to experts, let us give a proof for readers' convenience.

\begin{proposition}\label{gap4}
If
\begin{equation}
n=4,\quad \mathrm{Ric}\equiv 0,\quad \mathrm{AVR}>\frac{1}{2},
\end{equation}
then $M^4$ is isometric to $\mathbb{R}^4$.
\end{proposition}
\begin{proof}
Take a tangent cone at infinity, denoted by $C(Z^3)$, where $C(Z^3)$ is the metric cone over $Z^3$ by a result of Cheeger-Colding \cite{CheegerColding}. Thanks to a result of Cheeger-Naber \cite{CheegerNaber}, $Z^3$ is smooth Einstein with $\mathrm{Ric}\equiv 3$, up to a codimension $4$ singular set which thus must be empty since $Z^3$ is 3-dimensional. Thus $Z^3$ is smooth with $\mathrm{Ric}\equiv 2$, since $\mathrm{dim}(Z^3)=3$ again,  $Z^3$ has  constant sectional curvature $1$. Thus $Z^3$ is isometric to $\mathbb{S}^3/\Gamma$ for some finite subgroup $\Gamma$ of $O(4)$. On the other hand, $\mathrm{AVR}>\frac{1}{2}$ implies
\begin{equation}
\mathrm{vol}(Z^3) >\frac{1}{2}\mathrm{vol}(\mathbb{S}^3).
\end{equation}
This also implies $\sharp \Gamma<2$, namely $\Gamma$ is trivial. Thus $Z^3$ is isometric to $\mathbb{S}^3$. Therefore the conclusion follows from Theorem \ref{ChC}.
\end{proof}
The theorem above is sharp in the sense that the Eguchi-Hanson metric on $T\mathbb{S}^2$, which is Ricci flat, satisfies $\mathrm{AVR}=\frac{1}{2}$. Thus:
\begin{corollary}
It holds
\begin{equation}
C_{A, 0}(4)=\frac{1}{2}.
\end{equation}
\end{corollary}
Note that for higher dimensions, taking the product of $\mathbb{R}^k$ and the Eguchi-Hanson metric, it holds that
\begin{equation}
C_{A, 0}(n+1)\ge C_{A, 0}(n) \ge \frac{1}{2}, \quad \forall n \ge 5. 
\end{equation}

Similarly, we can also obtain the following result.
\begin{proposition}\label{gaptang}
Let 
\begin{equation}
(M_i^n, x_i) \stackrel{\mathrm{pGH}}{\to} (X^n, x) 
\end{equation}
be a non-collapsed pointed Gromov-Hausdorff convergent sequence of pointed Riemannian manifolds $M^n_i$ with $|\mathrm{Ric}| \le n-1$. Assume that a tangent cone at $x$ splits off $\mathbb{R}^{n-4}$.
Then $x$ is a regular point if and only if $D_n(x)>\frac{1}{2}$ holds, where the $n$-dimensional volume density $D_n(x)$ is defined by
\begin{equation}
D_n(x):=\lim_{r \to 0}\frac{\mathrm{vol}(B_r(x))}{\omega_nr^n}.
\end{equation}
\end{proposition}

\subsection{Gap metric-rigidity for positive Einstein manifolds}
In this subsection we deal with a positive Einstein analogue of the previous section. Firstly let us recall the following fundamental result in Riemannian geometry, which follows by analyzing the equality case in the Bishop-Gromov inequality.
\begin{theorem}
If
\begin{equation}
    \mathrm{Ric} \ge n-1,
\end{equation}
then
\begin{equation}\label{BM}
    \mathrm{vol} \le \mathrm{vol}(\mathbb{S}^n).
\end{equation}
Moreover the equality in (\ref{BM}) holds if and only if $M^n$ is isometric to $\mathbb{S}^n$.
\end{theorem}

\begin{theorem}[Cheeger-Colding]
    For any $n \ge 2$, there exists $\epsilon_n>0$ such that if 
    \begin{equation}
        \mathrm{Ric}\ge n-1,\quad \mathrm{vol}\ge (1-\epsilon_n)\mathrm{vol}(\mathbb{S}^n),
    \end{equation}
    then $M^n$ is diffeomorphic to $\mathbb{S}^n$.
\end{theorem}

The following, proved in \cite{HonM} by Mondello and the first named author, gives a positive Einstein analogue of Theorem \ref{Andersongap}.
\begin{theorem}
For any $n \ge 2$, there exists $\epsilon_n>0$ such that if 
    \begin{equation}
        \mathrm{Ric}\equiv n-1,\quad \mathrm{vol}\ge (1-\epsilon_n)\mathrm{vol}(\mathbb{S}^n),
    \end{equation}
    then $M^n$ is isometric to $\mathbb{S}^n$.
\end{theorem}
\begin{proof}
Let us provide a proof along \cite{HonM}, where another proof can be found by applying the local analytic structure of the moduli space of Einstein metrics, due to \cite[Chapter 12]{Besse} by Besse.

The proof is obtained by a contradiction argument. If it is not the case, then applying \cite{CheegerColding1, Coldingsphere}, there exists a sequence of closed manifolds $M^n_i$ with $\mathrm{Ric}\equiv n-1$ and $\mathrm{vol}(M_i^n) \to \mathrm{vol}(\mathbb{S}^n)$ such that $M_i^n$ smoothly converge to $\mathbb{S}^n$. In particular the curvature operator of $M_i^n$ is positive for any sufficiently large $i$.
Then, a result of Tachibana in \cite{Tachibana} yields that $M_i^n$ has  constant sectional curvature $1$. Since $M_i^n$ is diffeomorphic to $\mathbb{S}^n$, $M_i^n$ must be isometric to $\mathbb{S}^n$. This is a contradiction.
\end{proof}
Based on this result, let us define the Anderson constant for positive Einstein manifolds as follows.
\begin{definition}
    For any $n \ge 2$, define $C_{A, 1}(n)$ as the infimum of $\epsilon>0$ satisfying that if a closed Einstein manifold $M^n$ with $\mathrm{Ric}\equiv n-1$ satisfies $\mathrm{vol}>\epsilon \, \mathrm{vol}(\mathbb{S}^n)$, then $M^n$ is isometric to $\mathbb{S}^n$. 
\end{definition}
The  following question is rather  natural.
\begin{question}\label{andconst}
Determine $C_{A, 1}(n)$ explicitly.
\end{question}
Note that considering the standard projective space of dimension $n$, we know
\begin{equation}\label{eq:A1}
    C_{A, 1}(n)\ge \frac{1}{2},\quad \forall n \ge 2.
\end{equation}
Let us provide a simple observation in dimension $3$, which is already essentially observed in the proof of Theorem \ref{gap4}. 
\begin{proposition}\label{posi}
    If
    \begin{equation}\label{dimension3}
        n=3,\quad \mathrm{Ric}\equiv  2,\quad \mathrm{vol}>\frac{1}{2} \mathrm{vol}(\mathbb{S}^3),
    \end{equation}
    then $M^3$ is isometric to $\mathbb{S}^3$. With (\ref{eq:A1}), 
    \begin{equation}
        C_{A, 1}(3)=\frac{1}{2}.
    \end{equation}
\end{proposition}
\begin{proof}
    Since we are in dimension $3$, $M^3$ has constant sectional curvature $1$, namely $M^3$ is isometric to $\mathbb{S}^3/\Gamma$ for some finite subgroup $\Gamma$ of $O(4)$. However, the last assumption in (\ref{dimension3}) shows $\sharp \Gamma <2$, namely $\Gamma$ is trivial. Thus we conclude.
\end{proof}
Finally, we ask the following question.
\begin{question}\label{Ap}
    Is it true that
    \begin{equation}\label{aap}
        C_{A, 0}(n+1)=C_{A, 1}(n)
    \end{equation}
    for any $n \ge 3$?
\end{question}
As already observed, we know that (\ref{aap}) is correct if $n=3$, and that (\ref{aap}) is incorrect if $n=2$.

\begin{remark}\label{FS}
It is natural to ask whether $C_{A, 1}(4)=\frac{1}{2}$ holds. In fact,  it was pointed out by Shengxuan Zhou to the authors that this fails due to the Fubini-Study metric on $\mathbb{CP}^2$. The precise description is as follows.

Let us denote by $g_{{\rm FS}}$ the Fubini-Study metric on $\mathbb{CP}^2$, constructed by the Hopf fibration:
\begin{equation}
\mathbb{S}^1\to \mathbb{S}^5 \to \mathbb{CP}^2,
\end{equation}
where each fiber is a  great circle in $\mathbb{S}^5$. It is possible to compute the volume by 
\begin{equation}
    \mathrm{vol}(\mathbb{CP}^2)=\frac{\mathrm{vol}(\mathbb{S}^5)}{2\pi}=\frac{1}{2}\pi^2.
\end{equation}
On the other hand, it is well-known that $g_{{\rm FS}}$ is a positive Einstein metric whose Einstein constant is equal to $6$.

Therefore, considering the normalization $2\, g_{{\rm FS}}$, we obtain an Einstein $4$-manifold with $\mathrm{Ric}\equiv 3$ and $\mathrm{vol}=2\pi^2$. Since $\mathrm{vol}(\mathbb{S}^4)=(8\pi^2)/3$, the above discussion yields
\begin{equation}
    C_{A, 1}(4)\ge \frac{3}{4} \left(=2\pi^2 \cdot \frac{3}{8\pi^2}\right).
\end{equation}
\end{remark}

In connection with Question \ref{Ap}, let us provide two partial results. 

\begin{proposition}\label{gap5}
If
\begin{equation}
n=5,\quad \mathrm{Ric}\equiv 0,\quad \mathrm{AVR}>C_{A, 1}(4)
\end{equation}
then $M^5$ is isometric to $\mathbb{R}^5$. In other words,
\begin{equation}
C_{A, 0}(5)\le C_{A, 1}(4).
\end{equation}
\end{proposition}
\begin{proof}
Let us take a tangent cone at infinity, denoted by $C(Z^4)$, where $Z^4$ is smooth $4$-dimensional Einstein with $\mathrm{Ric}\equiv 3$ on the regular set because of \cite{CheegerColding1}. Firstly let us prove that $Z^4$ is smooth. It is enough to show that no singular point exists.
 
Take a point $z \in Z^4$ and a tangent cone at $z$ of $Z^4$, denoted by $C(W^3)$. Applying \cite{CheegerColding1} again, we know that $W^3$ is smooth $3$-dimensional Einstein with $\mathrm{Ric}\equiv 2$ on the regular set. Since $W^3$ is $3$-dimensional, then $W^3$ has no singular points because of the same reason as observed in the proof of Proposition \ref{gap4}, due to \cite{CheegerNaber}.

On the other hand, we regard $z$ as a point in $\partial B_1(p)$, where $p$ is the pole of $C(Z^4)$, and we consider a minimal geodesic $\gamma$ from $z$ to $p$ in $C(Z^4)$. Note that the $5$-dimensional volume density $D_5$
is lower semicontinuous along $\gamma$ because of the Bishop-Gromov inequality. Since $D_5$ is constant on $\gamma \setminus \{p\}$ and $D_5(p)=\mathrm{AVR}>C_{A, 1}(4)\ge \frac{1}{2}$,  we have
\begin{equation}
D_5(z)>C_{A, 1}(4)\ge \frac{1}{2}
\end{equation}
which shows that the volume of $W^3$ is greater than $\frac{1}{2}\mathrm{vol}(\mathbb{S}^3)$. Therefore Proposition \ref{posi} shows that $W^3$ is isometric to $\mathbb{S}^3$. Thus $z$ is a regular point, namely $Z^4$ is smooth.

Moreover the observation above also allows us to conclude that the volume of $Z^4$ is greater than $C_{A, 1}(4)\cdot\mathrm{vol}(\mathbb{S}^4)$. Thus by definition of $C_{A, 1}(4)$, $Z^4$ is isometric to $\mathbb{S}^4$. Therefore we conclude by Theorem \ref{rigid}. 

\end{proof}
The next proposition tells us that there are not so many examples of positive Einstein $4$-manifolds whose volume is greater than the half of the volume of the unit $4$-sphere.
\begin{proposition}
    For any $\epsilon>0$, the set of all Einstein manifolds $M^4$ of dimension $4$ with $\mathrm{Ric}\equiv 3$ and 
    \begin{equation}\label{vollower}
    \mathrm{vol} \ge (1+\epsilon)\cdot \frac{1}{2} \mathrm{vol}(\mathbb{S}^n)
    \end{equation}
    is compact with respect to the $C^{\infty}$-convergence.
\end{proposition}
\begin{proof}
Take a sequence $M^4_i$ of Einstein manifolds with $\mathrm{Ric}\equiv 3$ and satisfying (\ref{vollower}).  After passing to a subsequence, we have
\begin{equation}\label{GH}
    M_i^4 \stackrel{\mathrm{GH}}{\to} X^4
\end{equation}
for some compact non-collapsed Ricci limit space $X^4$. Since $X^4$ also satisfies (\ref{vollower}), in particular $D_4(x)>\frac{1}{2}$ by the Bishop-Gromov inequality. Thus $X^4$ has no singular points by Proposition \ref{gaptang} (or by just applying \cite{Anderson, Anderson2} by Anderson). It follows from \cite{CheegerColding1} that (\ref{GH}) can be improved to the smooth convergence.
\end{proof}
See, for instance, \cite{SX} for an analogous gap metric-rigidity result for negative Einstein manifolds.

\section{Gap topological-rigidity under lower curvature bounds}

\subsection{Lower sectional curvature bounds}

Let us start by recalling a classical result due to Marenich and Topogonov \cite{MarTop85Rus} (see also \cite{MarTop91Eng}).

\begin{theorem}[Marenich-Topogonov]\label{thm:MarTop}
Let $M^n$ be a complete Riemannian  $n$-dimensional manifold with   $$\mathrm{sec}\geq 0 \text{ and } \mathrm{AVR}> 0.$$
Then $M^n$ is diffeomorphic to $\mathbb{R}^n$.
\end{theorem}

To put Theorem \ref{thm:MarTop} in perspective, recall that Cheeger-Gromoll-Perelman soul Theorem implies that if  $M^n$ is a complete non-compact $n$-dimensional Riemannian manifold with non-negative sectional curvature, admitting a point $\bar{x}$ such that \emph{all} sectional curvatures at $\bar{x}$ are positive, then $M^n$ is diffeomorphic to $\mathbb{R}^n$.

Note also that Theorem \ref{thm:MarTop} does not generalise to the non-smooth setting of Alexandrov spaces with non-negative curvature. Indeed the Euclidean cone over $\mathbb{RP}^2$, that we denote by $C(\mathbb{RP}^2)$, is an example of a 3-dimensional Alexandrov space wih non-negative curvature, with  AVR $=1/2$, but it is not homeomorphic to $\mathbb{R}^3$. 

In Theorem \ref{thm:AlexRN}, we will show that as soon as AVR $>1/2$ then the topological rigidity holds.

The following result is a consequence of Grove-Petersen radius sphere theorem \cite{GrPet}.

\begin{theorem}\label{thm:AlexSN}
Let $X^N$ be an $N$-dimensional Alexandrov space with $$\mathrm{curv}\geq 1 \text{ and } \mathcal{H}^N(X^N)> \frac{1}{2}\mathrm{vol}(\mathbb{S}^N).$$
Then $X^N$ is homeomorphic to $\mathbb{S}^N$.
\end{theorem}

\begin{proof}
Bishop-Gromov inequality in Alexandrov spaces, combined with the assumption that $\mathcal{H}^N(X^N)\geq \mathcal{H}^N(\mathbb{S}^N)/2$ imply that the radius of $X^N$ is strictly larger than $\pi/2$. Grove-Petersen radius sphere theorem \cite{GrPet} gives that  $X^N$ is homeomorphic to $\mathbb{S}^N$.
\end{proof}

The next result is a consequence of Theorem \ref{thm:AlexSN}, the Bishop-Gromov inequality, and Perelman's stability theorem.

\begin{theorem}\label{thm:AlexRN}
Let $X^N$ be an $N$-dimensional Alexandrov space with $$\mathrm{curv}\geq 0 \text{ and } \mathrm{AVR}> 1/2.$$
Then $X^N$ is homeomorphic to $\mathbb{R}^N$.
\end{theorem}

\begin{proof}
Let $Y^N$ be a tangent cone at infinity of $X^N$. It is easily checked that also $Y^N$ has $\mathrm{AVR}> 1/2$.
\smallskip

\textbf{Step 1}. $Y^N$ is homeomorphic to $\mathbb{R}^N$.

We know that $Y^N=C(Z^{N-1})$ is the metric cone over an $(N-1)$-dimensional Alexandrov space $Z^{N-1}$ with $\mathrm{curv}\geq 1$. Moreover, since $Y^N$ has $\mathrm{AVR}> 1/2$, it follows that $\mathcal{H}^{N-1}(Z^{N-1})>\mathcal{H}^N(\mathbb{S}^{N-1})/2$. Therefore, Theorem  \ref{thm:AlexSN} gives that $Z^{N-1}$ is homeomorphic to $\mathbb{S}^{N-1}$. Thereofore, $Y^N=C(Z^{N-1})$ is homeomorphic to $\mathbb{R}^N$.
\smallskip

\textbf{Step 2}. $X^N$ is homeomorphic to $\mathbb{R}^N$.

Let $x_0\in X^N$. By Perelman's stability theorem, for any $\epsilon>0$, for every $R>0$ large enough, an open subset $U\subset X$, with $B_{(1-\epsilon)R}(x_0) \subset U \subset B_{(1+\epsilon)R}(x_0)$, is homeomorphic to $B_{R}(y_0)\subset Y$ which, in turn, is homeomorphic to $B_R\subset \mathbb{R}^N$, thanks to \textbf{Step 1}. 
At this point, arguing verbatim as in the proof of \cite[Theorem 3.5]{KapMon21} by Kapovitch and the second named author, it is possible to glue such local homeomorphisms into a global homeomorphism from $X^N$ to $\mathbb{R}^N$ by using Siebenmann’s deformation of homeomorphisms theory \cite{Sieb}.
\end{proof}

The constant $1/2$ in the AVR is sharp, as the aforementioned example of $C(\mathbb{RP}^2)$ shows. Note that  $C(\mathbb{RP}^2)\times\mathbb{R}^k$ gives an example in every dimension $n=3+k$.

 In case of a 2-dimensional Alexandrov space $X^2$ without boundary, with non-negative curvature and AVR $>0$, the classical Cohn-Vossen theorem implies that $X^2$ is homeomorphic to $\mathbb{R}^2$.

\subsection{Lower Ricci curvature bounds}
Throughout the section, we assume the reader is familiar with the notations and terminology of $\mathsf{RCD}(K,N)$ spaces, i.e. possibly non-smooth metric measure spaces with Ricci curvature bounded below by $K\in \mathbb{R}$ and dimension bounded above by $N\in [1,\infty]$ in a synthetic sense. Let us just recall that, thanks to the work of Cavalletti-Milman \cite{Cavalletti-Milman} (see also \cite{Li24} by Li) the conditions $\mathsf{RCD}(K,N)$ and  $\mathsf{RCD}^*(K,N)$ are equivalent, so we make no distinction between the two and adopt the notation $\mathsf{RCD}(K,N)$.

In connection with Theorem \ref{thm:MarTop} let us start by introducing a recent beautiful result due to Bru\`e-Pigati-Semola in \cite[Theorem 1.9]{BPS}. 
\begin{theorem}[Bru\`e-Pigati-Semola]
    Let $X^3$ be a non-collapsed $\mathsf{RCD}(0, 3)$ manifold with $\mathrm{AVR}>0$. Then $X^3$ is homeomorphic to $\mathbb{R}^3$. In particular any complete $3$-manifold with nonnegative Ricci curvature and $\mathrm{AVR}>0$ is diffeomorphic to $\mathbb{R}^3$.
\end{theorem}
The following is a consequence of the Bishop-Gromov inequality, Bru\`e-Pigati-Semola manifold recognition theorem for $\RCD$ spaces, and Perelman's proof of the Poincar\'e conjecture.   

\begin{theorem}\label{thm:RCDSN}
Let $N=2,3$ and let $(X^N, \mathsf{d}, \mathcal{H}^N)$ be a non-collapsed $\mathsf{RCD}(N-1, N)$ space with $\mathcal{H}^N(X)> \mathrm{vol}(\mathbb{S}^N)/2$.
Then $X^N$ is homeomorphic to $\mathbb{S}^N$.
\end{theorem}

\begin{proof}

\textbf{2-dimensional case}.  By Lytchack-Stadler \cite{LytStad}, $X^2$ is an Alexandrov space with $\mathrm{curv}\geq 1$. The claim then follows from Theorem \ref{thm:AlexSN}.
\smallskip

\textbf{3-dimensional case}
By Perelman's proof of the Poincar\'e conjecture, it  is enough to show that $X^3$ is homeomorphic to a simply connected topological 3-manifold without boundary.
\smallskip

\textbf{Step 1}. $X^3$ is homeomorphic to a topological 3-manifold.

For every $x\in X^3$, every tangent cone at $x$ is a metric cone $C(Y^2)$. By a result of Ketterer in \cite{Ketterer},  $Y^2$ is a non-collapsed $\mathsf{RCD}(1,2)$ space. The Bishop-Gromov volume monotonicity implies that $\mathcal{H}^2(Y^2)>\mathrm{vol}(\mathbb{S}^2)/2$. Thus, from the 2-dimensional case discussed above, it follows that $Y^2$ is homeomorphic to $\mathbb{S}^2$. Thus, we just proved that every tangent cone to every point of $X^3$ is homeomorphic to the cone over $\mathbb{S}^2$. The recent result proved in \cite[Theorem 1.8]{BPS} gives that $X^3$ is homeomorphic to a topological 3-manifold.
\smallskip

\textbf{Step 2}. $X^3$ is simply connected.

By a result of Wang in \cite{W}, $\mathsf{RCD}(K,N)$ spaces are semi-locally simply connected, thus the fundamental group coincides with the revised fundamental group. If $X^3$ is not simply connected, by the work of Wei and the second named author \cite{MW19}, we know that the universal cover $\tilde{X}^3$ of $X^3$ is an $\mathsf{RCD}(2,3)$ space, and it must hold $$\mathcal{H}^3(\tilde{X}^3)\geq 2 \mathcal{H}^3(X^3)> \mathrm{vol}(\mathbb{S}^3),$$
contradicting the Bishop-Gromov monotonicity. Thus $X^3$ is simply connected.
\smallskip

\textbf{Step 3}. Conclusion.
\\From \textbf{Steps 2} and \textbf{3} we know that $X^3$ is homeomorphic to a simply connected  $3$-dimensional topological manifold.  Perelman's proof of the Poincar\'e conjecture gives that $X^3$ is homeomorphic to $\mathbb{S}^3$.
\end{proof}

\begin{theorem}\label{thm:RCDRN}
Let $(X^N, \mathsf{d}, \mathcal{H}^N)$ be a non-collapsed $\mathsf{RCD}(0,N)$ space, for some $N\leq 4$, with $\mathrm{AVR}> 1/2$.
Then every tangent cone at infinity, and every tangent cone at any point,  are homeomorphic to $\mathbb{R}^N$.
\end{theorem}

\begin{proof}
    Let us discuss only the case $N=4$, the others being analogous (actually easier).

Let $Y^4$ be any tangent cone at infinity. Then $Y^4$ is an $\mathsf{RCD}(0,4)$ space  with $\mathrm{AVR}> 1/2$.  It is easily seen that $Y^4$ achieves equality in the Bishop-Gromov monotonicity and thus, by De Philippis-Gigli \cite{DePhGigGAFA} and \cite{Ketterer} by Ketterer, $Y^4=C(Z^3)$ is the  metric cone over a non-collapsed $\mathsf{RCD}(2, 3)$ space $Z^3$. 
By the Bishop-Gromov monotonicity, it is easily seen that $\mathcal{H}^3(Z^3)> \mathrm{vol}(\mathbb{S}^3)/2$. Recalling Theorem \ref{thm:RCDSN}, we infer that $Z^3$ is homeomorphic to $\mathbb{S}^3$. It follows that $Y^4=C(Z^3)$ is homeomorphic to $\mathbb{R}^4$, as claimed. 

Similarly we obtain the corresponding result for tangent cones at any point.
\end{proof}

\section{Open problems}
\subsection{Related to the Poincar\'e inequality}
Firstly we provide the following, related to Conjecture \ref{specgap}.
\begin{conjecture}
Let 
\begin{equation}\label{conv}
M^n_i \stackrel{\mathrm{GH}}{\to} X^n
\end{equation}
be a non-collapsed sequence of closed Riemannian manifolds $M^n_i$ with $\mathrm{Ric} \ge -(n-1)$, converging in GH-sense to a compact Ricci limit space $X^n$. Then
\begin{enumerate}
\item $b_1(M^n_i)=b_{H, 1}(X^n)$ for any sufficiently large $i$, where $b_{H, 1}(X^n)$ denotes the dimension of the space of harmonic $1$-forms in the sense of \cite{Gigli} by Gigli;
\item $b_{H,1}(X^n)=b_1(X^n)$.
\end{enumerate}
\end{conjecture}
Note that if the conjecture above is valid, then Conjecture \ref{specgap} is too, because of the same arguments as in the proof of Theorem \ref{gap4} (together with a result of Pan-Wei in \cite{PW}).
We provide a progress along this direction, yielding that the only remaining step is to prove $b_{H,1}(X^n) \le b_1(X^n)$.
\begin{proposition}
Under the same setting  (\ref{conv}) as in the conjecture above, it holds
\begin{equation}
b_1(M^n_i) \le b_{H,1}(X^n),\quad \text{for any sufficiently large $i$.}
\end{equation} 
In particular if $b_{H, 1}(X^n) \le b_1(X^n)$, then $b_{H,1}(X^n)=b_1(X^n)$.
\end{proposition}
\begin{proof}
Let us take a sequence of harmonic $1$-forms $\omega_i$ on $M_i^n$ with $\|\omega_i\|_{L^2}=1$. Thanks to a result of the first named author in \cite{Honda}, since $(\omega_i)_i$ has an $L^2$-strong convergent subsequence, with no loss of generality we can assume that $(\omega_i)_i$ is $L^2$-strongly converging to some $\omega \in L^2(T^*X^n)$. It follows from \cite{Honda} that $\omega \in W^{1,2}_H(T^*X^n) \cap L^{\infty}(T^*X^n)$ holds with $\di \omega=0$ and $\delta \omega=0$. Therefore in order to conclude, it is enough to check $\omega \in H^{1,2}_H(T^*X^n)$.

To this end, we use ideas appeared in a paper \cite{HKMPR} by Ketterer-Mondello-Perales-Rigoni and the first named author. We assume the readers are familiar with the standard notations about this topic, including the definitions of $\epsilon$-regular set $\mathcal{R}^n_{\epsilon}$ and $(\epsilon,k)$-singular stratum $ \mathcal{S}^k_{\epsilon}$. 

\textbf{Step 1.} There exists $\epsilon_n>0$ such that for any $x \in \mathcal{R}^n_{\epsilon_n}$ there exist $r>0$ and $f \in D(\Delta, B_r(x))$ such that $\omega=df$ on $B_r(x)$.

This is done by following the proof of \cite[Theorem 3.1]{HKMPR} based on the intrinsic Reifenberg method with the Poincar\'e lemma on $M_i^n$.

\textbf{Step 2.} For any Lipschitz function $\phi$ on $X^n$ with $\supp \phi \subset X^n \setminus \mathcal{S}^{n-2}_{\epsilon_n}$, we have $\phi \omega \in H^{1,2}_H(T^*X^n)$.

This is a direct consequence of \textbf{Step 1} and  a partition of unity argument: indeed,  writing $\omega=df$ on a small ball $B_r(x)$, we easily get that $\rho df \in H^{1,2}_H(T^*X^n)$, for any $\rho \in \mathrm{Lip}_c(B_r(x))$.

\textbf{Step 3.} Conclusion.

For any small $r>0$, find a cut-off Lipschitz function $\phi_r$ with $\supp (\phi_r) \subset X^n \setminus B_r(\mathcal{S}_{\epsilon}^{n-2})$, $\phi_r =1$ on $X^n \setminus B_{2r}(\mathcal{S}_{\epsilon}^{n-2})$ and $|\nabla \phi_r| \le \frac{2}{r}$. Recalling a result proved in \cite{CheegerJiangNaber} by Cheeger-Jiang-Naber:
\begin{equation}\label{eq:hausestn-2}
    \mathcal{H}^n(B_s(\mathcal{S}^{n-2}_{\epsilon})) \le C(n, \epsilon, \mathrm{diam}(X^n))s^2,\quad \forall s>0, 
\end{equation}
we have
\begin{equation*}
    \int_{X^n}|\nabla \phi_r|^2\di \mathcal{H}^n \le C(n, \epsilon, \mathrm{diam}(X^n)).
\end{equation*}
After choosing $\epsilon=\epsilon_n$ as in \textbf{Step 1}, letting $r \to 0$ for $ \phi_r \omega \in H^{1,2}_H(T^*X^n)$ with Mazur's lemma allows  to conclude that $1_{X^n \setminus \mathcal{S}^{n-2}_{\epsilon}}\omega \in H^{1, 2}_H(T^*X^n)$. Since $\mathcal{S}^{n-2}_{\epsilon}$ is $\haus^n$-negligible (for example recall (\ref{eq:hausestn-2})), we conclude. 
\end{proof}
\begin{question}
What is the sharp upper bound on $C_{P, 1}$ under $\mathrm{Ric} \ge K, \mathrm{diam} \le d$ and $\mathrm{vol} \ge v$?
\end{question}
\begin{question}\label{sharpconj}
    Prove Conjecture \ref{specgap} is sharp in the sense that there exists a sequence $M_i^n$ with $\mathrm{Ric}\ge -(n-1), \mathrm{diam} \le d$ and $\mathrm{vol}\ge v$ such that $C_{P, k}(M_i^n) \to \infty$ for some $k \ge 2$.
\end{question}
\begin{conjecture}\label{section}
    If 
    \begin{equation}\label{secpoinc}
        \mathrm{sec}\ge -1, \quad \mathrm{diam}\le d,\quad \mathrm{vol}\ge v,
    \end{equation}
    then for any $k \ge 1$
    \begin{equation}
        C_{P, k}\le C(n, d, k, v). 
    \end{equation}
\end{conjecture}

In connection with Conjecture \ref{section}, let us recall a renowned conjecture by Perelman.
\begin{conjecture}[Perelman's bi-Lipschitz stability conjecture]\label{perelman}
    Let
    \begin{equation}
        X_i^n \stackrel{\mathrm{GH}}{\to}X^n
    \end{equation}
    be a non-collapsed GH-convergent sequence of compact Alexandrov spaces\footnote{Note that, since $X^n$ is also assumed to be compact, then the diameters of $X_i^n$ are uniformly bounded.} of dimension $n$ with $\mathrm{curv}\ge -1$.  Then there exists $C>1$ such that $X_i^n$ is $C$-bi-Lipschitz equivalent to $X^n$ for any sufficiently large $i$.
\end{conjecture}
It is worth mentioning that if Conjecture \ref{perelman} holds, then, by a standard contradiction argument, the above $C>1$ can be taken as a constant depending only on the dimension $n$, a uniform upper bound of the diameters of $X_i^n$ and a uniform positive lower bound of $\mathrm{vol}(X_i^n)$.

The following gives a connection between the conjectures above.
\begin{proposition}
     Assume Conjecture \ref{perelman} holds in the special case when $X_i^n$ is a sequence of smooth Riemannian manifolds. Then Conjecture \ref{section} holds true.
\end{proposition}
\begin{proof}
The proof is done by a contradiction argument. If it is not the case, then there exists a sequence $M^n_i$ satisfying (\ref{secpoinc}) such that $M^n_i$ Gromov-Hausdorff converge to a compact Alexandrov space $X^n$ of dimension $n$ and that $C_{P, k}(M^n_i) \to \infty$. Thus $\mu_k(M_i^n) \to 0$. On the other hand, \cite[Lemma 4.2
]{L} by Lott yields
\begin{equation}
    \liminf_{i\to \infty}\mu_k(M_i^n)>0,
\end{equation}
which is a contradiction.
\end{proof}
In connection with the observation above, let us prove the following.
\begin{theorem}\label{perelman2}
Conjecture \ref{perelman} is valid if $n=2$ and each $X_i^2$ is smooth.
\end{theorem}

In order to prove this, let us prepare a couple of technical results. In the sequel we assume that the readers are familiar with the fundamental results  on Alexandrov surfaces $X^2$, including the curvature measure $\omega_{X^2}=\omega=\omega^+-\omega^-$. We refer to \cite{AZ, R, Mac,  Shioya, BL, BB, B} for more  details.

Firstly let us provide a quantitative way to construct a disc, whose boundary is a given (short) simple closed curve.
\begin{lemma}\label{lem1}
If
\begin{equation}\label{2}
n=2,\quad \mathrm{sec} \ge -1,\quad \mathrm{diam} \le d, \quad \mathrm{vol} \ge v
\end{equation}
then there exists $\tau=\tau(d, v)>0$ such that if a simple closed curve $\gamma$ in $M^2$ satisfies $\ell(\gamma) \le \tau$, then the interior $D_{\gamma}$ of $\gamma$ makes sense and it is a disc. In this case,  set  $D_{\gamma}^+:=M^2 \setminus (\gamma \cup D_{\gamma})$.
\end{lemma}
\begin{proof}
This is a direct consequence of Perelman's topological stability theorem \cite{Pere} (see also \cite{Kapo} by Kapovitch) and the Jordan curve theorem, via a contradiction argument. We omit the details.
\end{proof}
Theorem \ref{perelman2} is essentially a direct consequence of the following technical lemma, together with a result of Burago \cite{B}.
\begin{lemma}\label{1}
Under the same assumptions as in Lemma \ref{lem1}, there exist $\ell=\ell(d, v)>0$ and $\epsilon=\epsilon(d, v)>0$ such that if a simple closed curve $\gamma$ in $M^2$ satisfies $\ell(\gamma) \le \ell$, then $\gamma$ bounds a disc $D_{\gamma}$ with $\omega^+(D_{\gamma})\le 2\pi-\epsilon$.
\end{lemma}
\begin{proof}
The proof is done by a contradiction argument. If it is not the case, then there exist:
\begin{enumerate}
\item a sequence of closed smooth surfaces $M_i^2$ with curvature bounded below by $-1$, satisfying
\begin{equation}\label{3}
M_i^2 \stackrel{\mathrm{GH}}{\to} X^2
\end{equation}
for some compact Alexandrov space $X^2$ of dimension $2$;
\item  a sequence of simple closed curves $\gamma_i$ in $M_i^2$ with $\ell(\gamma_i) \to 0$ and $\omega_i^+(D_{\gamma_i}) \to 2\pi$. 
\end{enumerate}

In order to keep notation short, we assume that $M_i^2$ and $X^2$ are isometrically embedded in an ambient metric space $Z$ realising the GH-convergence, i.e.\;$M_i^2$ converge to $X^2$ in Hausdorff distance sense, as subsets of $Z$.

Up to a subsequence,  we can assume that $\gamma_i\to x_0 \in X^2$.
Let us divide the proof into the following cases:
\smallskip

\textit{(I) The case when $\limsup_{i\to \infty} \mathrm{diam} (D_{\gamma_i})>0$ and $\limsup_{i\to \infty} \mathrm{diam} (D_{\gamma_i}^+)>0$ .}

Then after passing to a subsequence we can find a sequence of points $x_i \in D_{\gamma_i}, y_i \in D_{\gamma_i}^+$ with $$\liminf_{i\to \infty}\dist_{M_i^2}(x_i, \gamma_i)>0 \quad \text{and}\quad\liminf_{i\to \infty}\dist_{M_i^2}(y_i, \gamma_i)>0.$$ With no loss of generality, we can assume $x_i \to x \in X^2$ and $y_i \to y \in X^2$ with respect to the metric of $Z$. Then by definition of $x_0$, for all $r<\min\{\dist_{X^2}(x, x_0), \dist_{X^2}(y, x_0)\}$,  $z \in B_r(x)$ and $w \in B_r(y)$, we see that any geodesic from $z$ to $w$ contains $x_0$ as an interior point.    This contradicts the non-branching property for geodesics in $X^2$, which is a consequence of the Alexandrov lower sectional curvature bound on $X^2$.
\smallskip

\textit{(II) The case when $\lim_{i\to \infty} \mathrm{diam} (D_{\gamma_i})=0$ and $\limsup_{i\to \infty} \mathrm{diam} (D_{\gamma_i}^+)>0$ .}

Fix $\phi \in C(X^2)$ with $0 \le \phi \le 1$, $\phi|_{B_{\frac{3r}{2}}(x_0)} \equiv 1$ and $\supp (\phi) \subset B_{2r}(x_0)$. For any $i$, fix $x_i \in \gamma_i$, and then find a uniformly convergent sequence $\phi_i \in C(M_i^2)$ to $\phi$ with $0 \le \phi_i \le 1$,  $\phi_i|_{B_r(x_i)}\equiv 1$ and $\supp \phi_i \subset B_{2r}(x_i)$. The weak convergence of the curvature measures gives
\begin{equation}\label{eq:Mitoomega}
\int_{M_i^2}\phi_i\di \omega_{M_i^2} \to \int_{X}\phi\di \omega_{X^2}=\int_{B_{2r}(x_0)}\phi\di \omega_{X^2}
\end{equation}
On the other hand, as $i \to \infty$,
\begin{align}\label{10}
\int_{M_i^2}\phi_i\di \omega_{M_i^2}&=\int_{B_{2r}(x_i)}\phi_iK_i^+\di \mathrm{vol}-\int_{B_{2r}(x_i)}\phi_iK_i^-\di \mathrm{vol} \nonumber \\
&\geq \left( \int_{D_{\gamma_i}}\phi_iK_i^+\di \mathrm{vol} \right) -  \mathrm{vol} (B_{2r}(x_i)) , \quad \text{(by $\ell(\gamma_i) \to 0$),} \nonumber \\
&\to 2\pi - \haus^2(B_{2r}(x_0)), \quad \text{(by $\omega_i^+(D_{\gamma_i}) \to 2\pi$),}
\end{align}
where, to obtain the second line,  we exploited the assumption that the sectional curvatures are bounded below by $-1$.
Also, it holds that
\begin{align}\label{eq:omega+phi}
\int_{B_{2r}(x_0)}\phi\di \omega_{X^2}&=\int_{B_{2r}(x_0)}\phi\di \omega^+_{X^2}-\int_{B_{2r}(x_0)}\phi\di \omega^-_{X^2} \le \int_{B_{2r}(x_0)}\phi\di \omega^+_{X^2} \nonumber\\ 
&\le \omega^+_{X^2}(B_{2r}(x_0)).
\end{align}
Combining  \eqref{eq:Mitoomega}, \eqref{10}, 
and \eqref{eq:omega+phi}, we obtain that
 $\omega^+_{X^2}(B_{2r}(x_0)) \to 2\pi$ as $r \to 0$, that is $\omega^+_{X^2}(\{x_0\})=2\pi$. As observed in \cite{Mac} (see also page 287 of \cite{AB} by Ambrosio-Bertrand), this is in contradiction with the fact that $X^2$ is 2-dimensional Alexandrov space with curvature bounded below. 
\smallskip

\textit{(III) Remaining cases.}
\\The remaining cases lead to a contradiction by analogous arguments. The proof is thus complete.
\end{proof}
We are now in a position  to prove Theorem \ref{perelman2}.

\begin{proof}[Proof of Theorem \ref{perelman2}.]
This is a direct consequence of \cite[Theorem 1]{B} and Lemma \ref{1}. 
\end{proof}

\subsection{Related to the gap metric-rigidity}

\begin{question}\label{EH}
    If a complete $4$-manifold $M^4$ with non-negative Ricci curvature satisfies $\mathrm{AVR}=\frac{1}{2}$, then is $M^4$ isometric to the Eguchi-Hanson metric? 
\end{question}
Note that the answer to Question \ref{EH} is affirmative if the Riemannian metric is K\"ahler.

\begin{question}
    It is possible to classify Einstein $4$-manifolds with $\mathrm{Ric}\equiv 3$ and $\mathrm{vol}\ge\mathrm{vol}(\mathbb{S}^4)/2$?
\end{question}

\subsection{Related to the gap topological-rigidity}

\begin{question}
Let $X^4$ be as in Theorem \ref{thm:RCDRN}. Is it true $X^4$ is homeomorphic to $\mathbb{R}^4$?
\end{question}

Even the smooth counterpart of such a question seems to be an interesting problem:

\begin{question}
Let $M^4$ be a 4-dimensional smooth complete Riemannian manifold with 
$$
\mathrm{Ric}\geq 0 \quad\text{and}\quad \mathrm{AVR}>1/2.
$$
 Is it true $M^4$ is homeomorphic to $\mathbb{R}^4$?  
 \\In the affirmative case, is $M^4$ diffeomorphic to $\mathbb{R}^4$? 
\end{question}

\begin{question}
For each $n\geq 4$, determine the minimal constant $C_{\geq,1}(n)$ with the following property.
If $M^n$ is an $n$-dimensional smooth closed Riemannian manifold with 
$$
\mathrm{Ric}\geq n-1 \quad\text{and}\quad \mathrm{vol}> C_{\geq,1}(n)\, \mathrm{vol}(\mathbb{S}^n),
$$
then $M^n$ is homeomorphic (resp.\,diffeomorphic) to $\mathbb{S}^n$.   
\end{question}
Note that Remark \ref{FS} implies that $C_{\geq,1}(4)\geq 3/4$.

Finally it is of course very interesting to investigate whether the corresponding non-smooth analogues of all of the above are valid or not.
\bigskip

\textbf{Acknowledgement.} This work was conducted during the event, \textit{School and Conference on Metric Measure Spaces, Ricci Curvature, and Optimal Transport}, at Villa Monastero, Lake Como. We would
like to express our gratitude to the organizers of the conference  for the beautiful and stimulating environment, and  to Gilles Carron for suggesting useful references. Moreover, we wish to thank John Lott for pointing out the current Proposition \ref{dim2poincare}, with its simple proof, and Shengxuan Zhou for pointing out Remark \ref{FS}. The first named author expresses his appreciation to Kota Hattori for his suggestions on the validity of Corollary \ref{cor:spectralgap4} for $2$-forms. Finally, we wish to thank the referee for his/her careful readying with valuable comments for a revision.

The first named author acknowledges support of the Grant-in-Aid for Scientific Research (B) of 20H01799, the Grant-in-Aid for Scientific Research (B) of
21H00977 and Grant-in-Aid for Transformative Research Areas (A) of 22H05105.

The second named author acknowledges support by the European Research
Council (ERC), under the European Union Horizon 2020 research and innovation programme, via the ERC Starting Grant “CURVATURE”, grant agreement
No. 802689.

\end{document}